\documentclass{article}
\usepackage{amsthm,amsmath,amssymb,amsfonts}
  \newtheorem{thm}{Theorem}[section]
  \newtheorem{lem}[thm]{Lemma}
  \newtheorem{prop}[thm]{Proposition}
  
  \theoremstyle{definition}
  
  \newtheorem{exm}[thm]{Example}
  \newtheorem{rmk}[thm]{Remark}
  
 \newcommand\ra{\rightarrow}

\newcommand{\lex}{\,\overrightarrow{\times}\,}
\newcommand{\RDP}{\mbox{\rm RDP}}
\newcommand{\RIP}{\mbox{\rm RIP}}

 \numberwithin{equation}{section}
\begin{document}
\title{When Lexicographic Product of Two po-Groups has the Riesz Decomposition Property}
%\title{ Riesz Decomposition Properties of po-Groups}
\author{ Anatolij Dvure\v{c}enskij$^{^{1,2}}$, Omid Zahiri$^{^{3}}$ \\
{\small\em $^1$Mathematical Institute,  Slovak Academy of Sciences,}\\
{\small\em\v Stef\'anikova 49, SK-814 73 Bratislava, Slovakia} \\
{\small\em $^2$Depart. Algebra  Geom.,  Palack\'{y} Univer., 17. listopadu 12,}\\
{\small\em CZ-771 46 Olomouc, Czech Republic} \\
{\small\em  $^{3}$University of Applied Science and Technology, Enghelab Av., Tehran, Iran}\\
{\small\tt  dvurecen@mat.savba.sk\quad   zahiri@protonmail.com} }
\date{}
\maketitle

\begin{abstract}
We study conditions when a certain type of the Riesz Decomposition Property (RDP for short) holds in the lexicographic product of two po-groups. Defining two important properties of po-groups, we extend known situations showing that the lexicographic product satisfies RDP or even RDP$_1$, a stronger type of RDP. We recall that a very strong type of RDP, RDP$_2$, entails that the group is lattice ordered.
RDP's of the lexicographic products are important for the study of lexicographic pseudo effect algebras, or perfect types of pseudo MV-algebras and pseudo effect algebras,  where infinitesimal elements play an important role both for algebras as well as for the first order logic of valid but not provable formulas.

\end{abstract}



{\small {\it Acknowledgement:} This work was supported by the grant VEGA No. 2/0069/16 SAV, and GA\v{C}R 15-15286S. }

\section{ Introduction }%1

In the last decades we observe that there is a growing interest to the study of some algebraic structures using lattice ordered groups or po-groups both for Abelian and non-Abelian ones. A prototypical situation is due to Mundici, see \cite{Mun, CDM}, when any MV-algebra is represented as an interval in a unital Abelian $\ell$-group. This result was extended in \cite{Dvu1} where there was proved that pseudo MV-algebras, a non-commutative generalization of MV-algebras, see \cite{GeIo, Rac}, can be represented by intervals in  unital $\ell$-groups not necessarily Abelian.

For mathematical foundations of quantum mechanics, Foulis and Bennett introduced in \cite{FoBe} effect algebras which are partial algebras with a partially defined operation $+$, where $a+b$ means disjunction of two mutually excluded events $a$ and $b$. These effect algebras are in many cases also intervals in Abelian po-groups (= partially ordered groups). A sufficient condition for such a po-group representation is the Riesz Decomposition property, RDP, of the effect algebra and of the po-group, as it follows from \cite{Rav}. RDP means roughly speaking a possibility to perform a joint refinement of any two decompositions of the same element, and po-groups with RDP are intensively studied in literature, see e.g. \cite{Fuc, Goo}. Recently effect algebras have been extended to non-commutative algebras, called pseudo effect algebras in \cite{DvVe1,DvVe2}. Also if such a pseudo effect algebra satisfies a stronger form of RDP, namely RDP$_1$, then the pseudo effect algebra is an interval in a po-group with RDP$_1$ not necessarily Abelian, see \cite{DvVe1,DvVe2}. If we define yet a more stronger type of RDP, RDP$_2$, then the corresponding pseudo effect algebra is even a pseudo MV-algebra.

A perfect MV-algebra is an MV-algebra where each element is either an infinitesimal or a co-infinitesimal. Di Nola and Lettieri \cite{DiLe1,DiLe2} showed that such MV-algebras can be represented as an interval in the lexicographic product $\mathbb Z \lex G$ with strong unit $(1,0)$ for some Abelian $\ell$-group $G$. Any perfect effect algebra with RDP was represented also as an interval in the lexicographic product $\mathbb Z \lex G$, where $G$ is an Abelian po-group with RDP, \cite{177}.

We note that perfect MV-algebras have no parallels in the realm of Boolean algebras because perfect MV-algebras are not semisimple. The logic of perfect MV-algebras has an analogue in the  Lindenbaum algebra of the first order \L ukasiewicz logic which is not semisimple, because the valid but unprovable formulas are precisely the formulas that correspond to co-infinitesimal elements of the Lindenbaum algebra, see e.g. \cite{DiGr}.

A more general type of MV-algebras which are intervals in the lexicographic products $A \lex G$, where $A$ is a linearly ordered group and $G$ is an $\ell$-group, were described in \cite{DFL}. Lexicographic types of pseudo MV-algebras were studied in \cite{283}, and lexicographic types of effect algebras were presented in \cite{277} as an interval in $A\lex G$, where $A$ is a linearly ordered group and $G$ is a po-group with RDP. As we see, for lexicographic pseudo MV-algebras or pseudo effect algebras, the lexicographic product of two po-groups with RDP play a crucial role, see also \cite{262, Dvu2015}, where it was shown that if $A$ is an antilattice po-group with RDP and $G$ a directed po-group with RDP, then $A\lex G$ has RDP.

Therefore, the study of perfect MV-algebras, perfect pseudo MV-algebras or perfect pseudo effect algebras is tightly connected with an important phenomenon of the first order \L ukasiewicz logic on one side and on the lexicographic product of two po-groups with some kind od RDP on the second side. This is for us a good excuse to study lexicographic product of po-groups.

Thus our main goal of the present paper is to extend situations when the lexicographic product $A \lex G$ satisfies RDP.

The paper is organized as follows. The second section is an introduction to theory of po-groups. Here we define some kinds of RDP's, we show their basic properties as well as a correct example that RDP $\nRightarrow$ RDP$_1$, which besides is of the form of the lexicographic product. The third section describes situations when $A\lex G$ satisfies RDP, if $A$ is not necessarily an antilattice po-group with RDP. In the fourth section we define the com-directness property of a directed po-group $G$ which is stronger than directness and for Abelian po-groups they are equivalent. We show that then $A\lex G$ will have RDP for each $A$ and $G$ with RDP. The fifth  section is a continuation of the research. We define Non-Comparability Directness Property and we show its importance for the lexicographic product. In Conclusion we summarize our results and indicate possible applications for lexicographic pseudo effect algebras and some open questions are presented.

\section{Riesz Decomposition Properties of po-Groups}%2

We remind that a {\it po-group} is an additively written group $(G;+,0)$ endowed with a partial order $\le$ such that $g\le h$ implies $a+g+b\le a+h+b$ for all $a,b,h,g \in G$. The {\it lexicographic product} of two po-groups $(G_1;+,0)$ and $(G_2;+,0)$ is the direct product $G_1\times G_2$ endowed with the lexicographic ordering $\le$ such that $(g_1,h_1) \le (g_2,h_2)$ iff $g_1<g_2$ or $g_1=g_2$ and $h_1\le h_2$ for $(g_1,h_1), (g_2,h_2)\in G_1\times G_2$. If the order $\le$ implies that $G$ is a lattice, we say that $G$ is a {\it lattice ordered} po-group, or simply an $\ell$-group.

We denote by $G^+:=\{g \in G \mid g\ge 0\}$ and $G^-=\{g \in G \mid g \le 0\}$. An element $u \in G^+$ is said to be a {\it strong unit} (or order unit) if, given $g \in G$, there is an integer $n\ge 0$ such that $g \le nu$. A {\it unital po-group} is a pair $(G,u)$, where $G$ is a po-group and $u$ is a fixed strong unit of $G$. If $(H,u)$ is a unital po-group, then $(H\lex G,(u,0))$ is a unital po-group.

For more information about po-groups we recommend for example the following books \cite{Dar,Fuc1, Gla}.

The {\it center} of a group $G$ is the set $\mbox{Z}(G)=\{x\in G\mid x+y=y+x$ for all $y\in G\}$.

A po-group $G$ is {\it directed} if, given $g_1,g_2 \in G$, there is $h \in G$ such that $g_1,g_2 \le h$. This is equivalent to the property: Given $g_1,g_2 \in G$, there is $h \in G$ such that $g_1,g_2 \ge h$; or equivalently $G^+-G^+=G$. %A subgroup $C$ of a po-group $G$ is {\it convex} if, given $c_1,c_2\in C$ and $g\in G$, $c_1\le g \le c_2$ implies $g \in C$.

A poset $(P;\le)$ is said to be an {\it antilattice} if only comparable elements $a,b\in P$ have a joint  or  meet in $P$. A directed po-group $G$ is an antilattice iff $a\wedge b =0$ implies $a=0$ or $b=0$. For example, antilattices are (i) every linearly ordered group, (ii) $\mathcal B(H)$, the group of Hermitian operators of a Hilbert space $H$, \cite[Thm 58.4]{LuZa}, (iii) $G=\mathbb R^2$ with the positive cone of all $(x,y)$
such that either $x = y = 0$ or $x > 0$ and $y > 0$; in addition, $G$ is an antilattice with RDP, but $G$ is not a lattice.

In the literature, see e.g. \cite{Fuc, Goo, DvVe1}, there is a whole variety of the Riesz Decomposition Properties.

We say that a po-group $(G;+,0)$ satisfies

\begin{enumerate}
\item[(i)]
the {\it Riesz Interpolation Property} (RIP for short) if, for $a_1,a_2, b_1,b_2\in G,$  $a_1,a_2 \le b_1,b_2$  implies there exists an element $c\in G$ such that $a_1,a_2 \le c \le b_1,b_2;$

\item[(ii)]
\RDP$_0$  if, for $a,b,c \in G^+,$ $a \le b+c$, there exist $b_1,c_1 \in G^+,$ such that $b_1\le b,$ $c_1 \le c$ and $a = b_1 +c_1;$

\item[(iii)]
\RDP\  if, for all $a_1,a_2,b_1,b_2 \in G^+$ such that $a_1 + a_2 = b_1+b_2,$ there are four elements $c_{11},c_{12},c_{21},c_{22}\in G^+$ such that $a_1 = c_{11}+c_{12},$ $a_2= c_{21}+c_{22},$ $b_1= c_{11} + c_{21}$ and $b_2= c_{12}+c_{22};$

\item[(iv)]
\RDP$_1$  if, for all $a_1,a_2,b_1,b_2 \in G^+$ such that $a_1 + a_2 = b_1+b_2,$ there are four elements $c_{11},c_{12},c_{21},c_{22}\in G^+$ such that $a_1 = c_{11}+c_{12},$ $a_2= c_{21}+c_{22},$ $b_1= c_{11} + c_{21}$ and $b_2= c_{12}+c_{22}$, and $0\le x\le c_{12}$ and $0\le y \le c_{21}$ imply  $x+y=y+x;$

\item[(v)]
\RDP$_2$  if, for all $a_1,a_2,b_1,b_2 \in G^+$ such that $a_1 + a_2 = b_1+b_2,$ there are four elements $c_{11},c_{12},c_{21},c_{22}\in G^+$ such that $a_1 = c_{11}+c_{12},$ $a_2= c_{21}+c_{22},$ $b_1= c_{11} + c_{21}$ and $b_2= c_{12}+c_{22}$, and $c_{12}\wedge c_{21}=0.$

\end{enumerate}

If, for $a,b \in G^+,$ we have for all $0\le x \le a$ and $0\le y\le b,$ $x+y=y+x,$ we denote this property by $a\, \mbox{\rm \bf com}\, b.$

The RDP will be denoted by the following table:

$$
\begin{matrix}
a_1  &\vline & c_{11} & c_{12}\\
a_{2} &\vline & c_{21} & c_{22}\\
  \hline     &\vline      &b_{1} & b_{2}
\end{matrix}\ \ .
$$

For Abelian po-groups, RDP, RDP$_1,$ RDP$_0$ and RIP are equivalent.

By \cite[Prop 4.2]{DvVe1}, for directed po-groups we have
$$
\RDP_2 \quad \Rightarrow \RDP_1 \quad \Rightarrow \RDP \quad \Rightarrow \RDP_0 \quad \Leftrightarrow \quad  \RIP, \eqno(2.1)
$$
but the converse implications do not hold, in general. More precisely, in \cite[Prop 4.2(ii)]{DvVe1}, there was proved (i) a directed po-group $G$ satisfies \RDP$_2$ iff $G$ is an $\ell$-group, and in general, (ii) RDP$_1 \not\Rightarrow \RDP_2$, (iii) $\RDP_0\not\Rightarrow \RDP$.

\begin{rmk}\label{rm:RDP}
In \cite[Prop 4.2(ii)]{DvVe1}, there was found an example of a po-group with RDP but RDP$_1$ fails in it.  Unfortunately, as we now show, \cite[Ex 3.5]{DvVe1} has a gap because, the po-group from that example is in fact Abelian, so it satisfies both RDP as well as RDP$_1$.
\end{rmk}

\begin{proof}
Let $G$ be an additive group generated by the countably many elements $g_0, g_1, \ldots$, let $v: (G;+,0)\to (\mathbb R;+,0)$ be the homomorphism determined by the conditions $v(g_i ) = (1/2)^i$, $i =0,1,\ldots$, and let $G$ fulfil the condition that every $a \in G$ such that
$v(a) = 0$ commutes with each element $b \in G$. Define a partial order in $G$ by setting
$G^+:= \{x \in  G\mid x = 0 \mbox{ or } v(x) > 0\}$. This means that we have for $a, b \in G$ $a\le b$ iff $a=b$ or $v(a)<v(b)$. In \cite[Ex 3.5]{DvVe1}, there was proved that $G$ satisfies RDP.

In what follows, we show that every $g_i$ commutes with each $g_j$ for $i,j=0,1,\ldots$. Indeed, let $i<j$. For the element $g_i -2^{j-i}g_j$, we have $v(g_i -2^{j-i}g_j)=0$, so that it commutes with every element of $G$, in particular with $g_j$. Then $g_j +(g_i -2^{j-i}g_j)=(g_i -2^{j-i}g_j)+g_j= g_i+g_j - 2^{j-i}g_j$ so that $g_j+g_i=g_j+g_i$. Since the set $\{g_0,g_1,\ldots\}$ generates $G$, $G$ is Abelian. Whence, it satisfies RDP$_1$, which contradicts the statement in \cite[Ex 3.5]{DvVe1}.
\end{proof}

The following example is a simple one showing how we can create non-Abelian po-groups $H\lex G$ with RDP or RDP$_1$.

\begin{exm}\label{RDP1}
Let $(G;+,0)$ be a po-group with RDP (RDP$_1$) and $(A;+,0)$ be a non-Abelian group. Consider the trivial ordering on $A$. Then  $A\lex G$ is a non-Abelian po-group. It can be easily seen that $A\lex G$ has RDP (RDP$_1$). Indeed, for each positive element $(a_1,a_2),(b_1,b_2),(c_1,c_2),(d_1,d_2)\in A\lex G$,
such that $(a_1,a_2)+(b_1,b_2)=(c_1,c_2)+(d_1,d_2)$ we have $a_1=b_1=c_1=d_1=0$ and so we can find an RDP table (RDP$_1$ table) for it.

In particular, if $G$ is in addition with RDP, then $A\lex G$ is a non-Abelian po-group satisfying RDP$_1$.
\end{exm}

The following example shows that there exists a po-group such that RDP $\nRightarrow$ RDP$_1$. This example corrects the implication in the example of \cite[Prop 4.2(ii), Ex. 3.5]{DvVe1}, see Remark \ref{rm:RDP}.

First we present the following definition.

Let $\{A_i\mid i \in I\}$ be a system of po-groups. The direct product $\prod_{i\in I}A_i$ consists of elements of the form $(a^i)_{i\in I}$ or simply $(a^i)$, where $a^i \in A_i$ for all $i \in I$.
Besides the product ordering $\le$ on the direct product $\prod_{i \in I}A_i$, defined by $(a^i)_{i\in I}\le (b^i)_{i\in I}$ iff $a^i\le b^i$ for all $i \in I$, we define the {\it strict product ordering}, $\leqq$, defined by $(a^i)_{i\in I}\leqq (b^i)_{i\in I}$ iff either $a^i = b^i$ for all $i \in I$ or $a^i <b^i$ for all $i \in I$.

\begin{lem}\label{le:RxR}
Consider the po-group $\mathbb{R}\times \mathbb{R}$ with the strict product ordering $\leqq$. Let $(G;+,0)$ be a non-Abelian
directed po-group with \RDP. Then $(\mathbb R\times \mathbb R)\lex G$ satisfies \RDP. If $G$ satisfies \RDP$_1$, then $((\mathbb R\times \mathbb R)\lex G, ((1,1),0))$ is a unital po-group satisfying \RDP,  but \RDP$_1$ fails for it.
\end{lem}

\begin{proof}
Let $((a_1,a_2),x), ((b_1,b_2),y), ((c_1,c_2),z), ((d_1,d_2),u)\ge 0$ in $(\mathbb R\times \mathbb R)\lex G$ be such that $((a_1,a_2),x)+ ((b_1,b_2),y)= ((c_1,c_2),z)+ ((d_1,d_2),u)$. Then either $a_1,a_2>0$ or $a_1,a_2=0$, and similarly for the elements $(b_1,b_2), (c_1,c_2), (d_1,d_2)$.

Inasmuch as $\mathbb R$ is a linearly ordered group, from \cite[Thm 3.1]{Dvu2015}, we conclude that, for $(a_1,x)+(b_1,y)=(c_1,z)+(d_1,u)$, there is an RDP decomposition

\[
\begin{matrix}
(a_1,x) &\vline & (n'_{11},c_{11}) & (n'_{12},c_{12})\\
(b_1,y) &\vline & (n'_{21},c_{21}) & (n'_{22},c_{22})\\
  \hline     &\vline      &(c_i,z) & (d_i,u)
\end{matrix}\
\]
and an RDP decomposition for $a_2+b_2=c_2+d_2$

\[
\begin{matrix}
a_2 &\vline & n''_{11} & n''_{12}\\
b_2 &\vline & n''_{21} & n''_{22}\\
  \hline     &\vline      &c_2 & d_2
\end{matrix}\ .
\]
Put

\[
\begin{matrix}
((a_1,a_2),x) &\vline & ((n'_{11},n''_{11}),c_{11}) & ((n'_{12},n''_{12}),c_{12})\\
((b_1,b_2),y) &\vline & ((n'_{21},n''_{21}),c_{21}) & ((n'_{22},n''_{22})_{i \in I},c_{22})\\
  \hline     &\vline      &((c_1,c_2),z) & ((d_1,d_2,u)
\end{matrix}\ . \ (A)
\]

If, one pair of the elements, e.g. $(a_1,a_2)=(0,0)$, then $(n'_{11},n''_{11})=(0,0)=(n'_{12},n''_{12})$, and for this case, (A) gives an RDP table.

Therefore, we can assume that $a_1,a_2,b_1,b_2,c_1,d_2>0$. We claim that
in the table (A) we can assume that all $n'_{ij}, n''_{ij}>0$ for $i,j=1,2$. Indeed, take e.g. the table

\[
\begin{matrix}
a_1 &\vline & n'_{11} & n'_{12}\\
b_i &\vline & n'_{21} & n'_{22}\\
  \hline     &\vline      &c_i & d_i
\end{matrix}\ \ .
\]

If, say $n'_{11}=0$, then $n'_{12}>0$ and $n'_{21}>0$. Since $n'_{12}$ and $n'_{21}$ are comparable, there is $n_0 \in A_i$ such that $0<n'_0 < n'_{12},n'_{21}$. Then in the table

\[
\begin{matrix}
a_1 &\vline & n'_{11}+n'_0 & n'_{12}-n'_0\\
b_1 &\vline & n'_{21}-n'_0 & n'_{22}+n'_0\\
  \hline     &\vline      &c_1 & d_1
\end{matrix}\ \
\]
all entries are strictly positive. In the same way we proceed with $n''_{ij}$.

Now we show that if $G$ satisfies RDP$_1$, then it does not hold in $(\mathbb R\times \mathbb R)\lex G$. Put $A = \mathbb R \times \mathbb R$ and let $\pi_1:A\lex G\ra A$ be the canonical projection map.
Since $G$ is non-Abelian, there are $x,y\in G$ such that $x+y\neq y+x$. Put $a,b\in G$ such that
$x+y=a+b$. Consider the equation
$((1,4),x)+((3,7),y)=((2,3),a)+((2,8),b)$. For each RDP table

$$
\begin{matrix}
((1,4),x)  &\vline & c_{11} & c_{12}\\
((3,7),y) &\vline & c_{21} & c_{22}\\
  \hline     &\vline      &((2,3),a) & ((2,8),b)
\end{matrix}\ \ ,
$$
we have $(0,0)\leqq \pi_1(c_{12})$ and  $(0,0)\leqq \pi_1(c_{21})$.
Since $(1,4)$ and $(2,3)$ are not comparable as well as $(3,7)$ and $(2,3)$, then $\pi_1(c_{12})$ and $\pi_1(c_{21})$ are non-zero elements of $A$
and so $\pi_1(c_{12})$ and $\pi_1(c_{21})$ are strictly positive. We can select a strictly positive element $(s,t)$ of $A$ such that
$(s,t)\leqq \pi_1(c_{12})$, $(s,t)\leqq \pi_1(c_{21})$ and $(s,t)\ne \pi_1(c_{12})$, $(s,t)\ne\pi_1(c_{21})$. Clearly,
$((s,t),x)\leq c_{12}$ and $((s,t),y)\leq c_{21}$ and $((s,t),y)+((s,t),x)\neq ((s,t),x)+((s,t),y)$.
\end{proof}

We note that for Abelian po-groups the equivalence RIP, RDP$_0$ and RDP was established in \cite[Prop 2.1]{Goo} without assumption that $G$ is directed. In \cite[Prop 4.2]{DvVe1}, the implications RDP $\Rightarrow$ RDP$_0$ $\Leftrightarrow$ RIP was proved for all po-groups under the assumption $G$ is directed. In what follows, we prove that RDP$_0$ is equivalent to RIP for any po-group $G$ not assuming $G$ is directed.

\begin{lem}\label{le:RIP}
In any po-group $G$, {\rm RDP}$_0$ is equivalent to \RIP.
\end{lem}

\begin{proof}
Let $G$ satisfy RDP$_0$ and let $a_1,a_2\le b_1,b_2$. Then $b_i-a_j \ge 0$ for $i,j=1,2$, and
$$
b_2-a_1 = (b_2-a_2)+(a_2-b_1)+(b_1 - a_1)\le (b_2 - a_2)+(b_1-a_1).
$$

Due to RDP$_0$, there are $c_1,c_2 \in  G^+$ such that $b_2-a_1 = c_1 +c_2$ and $c_1 \le b_2-a_2$, $c_2 \le b_1-a_1$. If we put $c = c_2 + a_1$, we have $b_2 = c_1 +c_2 +a_1 = c_1 + c$ which entails $c \le b_2$ and $a_1\le c$. On the other hand, $c-a_1 = c_2 \le b_1 - a_1$ which gives $c \le b_1$. Finally, $b_2 = c_1 + c \le (b_2- a_2)+ c$, so that $a_2 \le c$. Hence, $G$ satisfies RIP.

The converse implication follows from \cite[Prop 4.2]{DvVe1}.
\end{proof}

According to the latter result, the assumption of directness of a po-group $G$ is superfluous in (2.1).

\begin{prop}\label{pr:RDP5} Let $A$ be an antilattice po-group satisfying \RDP\ \textup(\RDP$_1$\textup). If for $a_1,a_2,b_1,b_2\in A^+$ we have $a_1 +a_2 = b_1 +b_2$, where $a_1\| b_1$, then there is an \RDP \ \textup(\RDP$_1$\textup) decomposition $(n_{ij})$, $i,j=1,2$, in $A$ such $n_{12},n_{21}>0$. In addition, in such a case, $n_{12},n_{21}>0$ and $n_{12}\|n_{21}$.

Moreover, if $(m_{ij})$, $i,j=1,2$, is an arbitrary {\rm RDP} decomposition for $a_1+a_2 =b_1+b_2$, then $m_{12}>0$ and $m_{21}>0$.
\end{prop}

\begin{proof}
Since $a_1$ and $b_1$ are not comparable, so are $a_2$ and $b_2$, and hence $a_1,b_1>0$ as well as $a_2,b_2>0$. We assert that there is $n_0 \in A$ such that $0<n_0<a_1,b_1$. Suppose the converse. We show that then $a_1\wedge b_1$ exists in $A$ and $a_1\wedge b_1 =0$. Let $c\le a_1,b_1$. Since RDP entails the Riesz Interpolation Property, Lemma \ref{le:RIP}, there is $d\in A$ such that $c,0\le d\le a_1,b_1$. Due to the assumptions, $d<a_1,b_1$ and $d=0$, so that $c\le 0$, and $a_1\wedge b_1=0$ which is impossible because $a_1$ and $b_1$ are incomparable which proves the assertion.

Similarly, there is $m_0 \in A$ such that $0<m_0<a_2,b_2$. Hence, we have $(-n_0 +a_1)+(a_2-m_0)=(-n_0+b_1)+(b_2-m_0)$, where all the elements in brackets are strictly positive. Due to RDP (RDP$_1$) of $A$, we have an RDP (RDP$_1$) table for $(-n_0 +a_1)+(a_2-m_0)=(-n_0+b_1)+(b_2-m_0)$ as follows

\[
\begin{matrix}
-n_0+a_1 &\vline & n_{11} & n_{12}\\
a_2-m_0 &\vline & n_{21} & n_{22}\\
  \hline     &\vline      & -n_0+b_1 & b_2-m_0
\end{matrix}\ \ ,
\]
which gives
\[
\begin{matrix}
a_1 &\vline & n_0+n_{11} & n_{12}\\
a_2 &\vline & n_{21} & n_{22}+m_0\\
  \hline     &\vline      & b_1 & b_2
\end{matrix}\ \ ,
\]
where the elements in the upper left-side corner and in the lower right-side corner are strictly positive.

In other words, if $a_1\|b_1$, there is always an RDP  (RDP$_1$) table for $a_1+a_2=b_1+ b_2$
\[
\begin{matrix}
a_1 &\vline & n_{11} & n_{12}\\
a_2 &\vline & n_{21} & n_{22}\\
  \hline     &\vline      & b_1 & b_2
\end{matrix}\ \
\]
such that $n_{11},n_{22}>0$.

Assume that for our RDP table $n_{12}=0$. Then $n_{11}=a_1\le b_1$ which is impossible. In the similar way, we can prove that $n_{21}>0$.

Now let $n_{12}$ and $n_{21}$ be comparable. Due to the equality $n_{11}+n_{12}+n_{21}+n_{22} = n_{11}+n_{21}+n_{12}+n_{22}$, we have $n_{12}+n_{21}=n_{21}+n_{12}$.
If $n_{12}\le n_{21}$, then
\[
\begin{matrix}
a_1 &\vline & n_{11}+n_{12} & 0\\
a_2 &\vline & -n_{12}+n_{21} & n_{12}+n_{22}\\
  \hline     &\vline      & b_1 & b_2
\end{matrix}\ \
\]
and this is also an RDP table for $a_1+a_2=b_1+b_2$. But in such a case, $a_1\le b_1$ which is a contradiction. In a similar way we can prove that  $n_{21}\not\le n_{21}$. Hence, $n_{12}\| n_{21}$.

Similarly, if $(m_{ij})$ is a decomposition, then in the same way as for $(n_{ij})$ we have $m_{12}>0$ and $m_{21}>0$.
\end{proof}

\section{Riesz Decomposition Properties of the Lexicographic Product}%3

In this section, we concentrate to the Riesz Decomposition Properties of the lexicographic product of two po-groups. In particular, we introduce the Com-Directness Property for po-groups. We start with the following result which was proved in \cite[Thm 3.3]{Dvu2015}.

\begin{thm}\label{th:3.3}
Let $A$ be an antilattice po-group and $G$ be a directed po-group. Then $A\lex G$ satisfies \RDP\, if and only if both $A$ and $G$ satisfy \RDP.
\end{thm}

We note that we do not know whether Theorem \ref{th:3.3} holds without assumption $A$ is an antilattice po-group. In what follows we extend Theorem \ref{th:3.3}.

It is very interesting to mention that in Lemma \ref{le:RxR}, the po-group $A=\mathbb R \times \mathbb R$ is an Abelian antilattice po-group with RDP and as well as RDP$_1$, but as we have seen, if $G$ is a directed non-Abelian po-group with RDP$_1$, then $A \lex G$ has RDP but RDP$_1$ fails. So the assumption that $A$ is an antilattice is not a guarantee to be $A\lex G$ with RDP$_1$ if $G$ has RDP$_1$.

We remind that if $A$ is a linearly ordered group and $G$ is a directed po-group with RDP$_1$, then $A\lex G$ satisfies RDP$_1$, see \cite[Thm 3.3]{Dvu2015}.

The following result is motivated by Lemma \ref{le:RxR}.

\begin{thm}\label{th:RDP3}
Let $\{A_i\mid i \in I\}$ be a system of non-trivial linearly ordered groups such that, for each $i\in I$, if $a \in A_i$, such that $a>0$, then there is $a_0$ in $A_i$ with $0<a_0<a$. Let the direct product $A=\prod_{i\in I} A_i$ be endowed with the strict product ordering, and $G$ be a directed po-group with \RDP. Then $A\lex G$ has \RDP\, whenever every $A_i$ is Abelian.

If $G$ is a non-Abelian po-group  with \RDP$_1$ and every $A_i$ is Abelian, then $A\lex G$ has \RDP\,  but \RDP$_1$\, fails if $|I| >1$.

%\textup(\RDP$_1$\textup)
\end{thm}

\begin{proof}
First we show that $A = \prod_{i\in I} A_i$ is an antilattice po-group with RDP. Let $c=(c^i)_{i \in I}\in A$   be the infimum of two mutually non-comparable elements  $a=(a^i)_{i \in I}\in A$ and $b=(b^i)_{i \in I} \in A$. Then $c^i <a^i,b^i$ for each $i$. Since $a^i$ and $b^i$ are comparable, there is an element $d^i$ such that $c^i < d^i <a^i, b^i$ which says $c < (d^i)_{i \in I} <a,b$, an absurd.  Hence, $A$ is an antilattice.

Assume every $A_i$ is Abelian.
Let $a=(a^i)_{i \in I}\in A^+$, $b=(b^i)_{i \in I} \in A^+$, $c=(c^i)_{i \in I}\in A^+$, and $d=(d^i)_{i \in I}\in A^+$ be such that
$$
a+b=c+d.\eqno(3.0)
$$
Then every $a^i, b^i, c^i, d^i \ge 0$. If one of $a,b,c,d$ is zero, an RDP table for (3.0) is evident. So assume that every $a^i,b^i,c^i,d^i$ is strictly positive for each $i \in I$. Then using RDP in each $A_i$, we have an RDP table

\[
\begin{matrix}
a^i &\vline & n^i_{11} & n^i_{12}\\
b^i &\vline & n^i_{21} & n^i_{22}\\
  \hline     &\vline      &c^i & d^i
\end{matrix}\ \ (A) .
\]

If, say $n^i_{11}=0$, then $n^i_{12}>0$ and $n^i_{21}>0$. Since $n^i_{12}$ and $n^i_{21}$ are ordered, there is $n_0^i \in A_i$ such that $0<n_0^i < n^i_{12},n^i_{21}$. Then in the table

\[
\begin{matrix}
a^i &\vline & n^i_{11}+n^i_0 & n^i_{12}-n^i_0\\
b^i &\vline & n^i_{21}-n^i_0 & n^i_{22}+n^i_0\\
  \hline     &\vline      &c^i & d^i
\end{matrix}\ \
\]
all entries are strictly positive, so that we can assume that all entries in (A) are strictly positive.

\[
\begin{matrix}
(a^i)_{i \in I} &\vline & (n_{11}^i)_{i \in I} & (n_{12}^i)_{i \in I}\\
(b^i)_{i \in I} &\vline & (n_{21}^i)_{i \in I} & (n_{22}^i)_{i \in I}\\
  \hline     &\vline      &(c^i)_{i \in I} & (d^i)_{i \in I}
\end{matrix}\ \
\]
is an RDP table for (2.1).

Since $A$ is an antilattice with RDP, by \cite[Thm. 3.3]{Dvu2015}, $A\lex G$ has RDP.

Now assume that $G$ is a non-Abelian po-group with RDP$_1$ and let $|I|>1$. By the first part of the present proof, $A\lex G$ has RDP. In what follows, we show that RDP$_1$ fails.

Since, every $A_i$ is non-trivial, it has infinitely many elements, and the index set has at least two elements, fix $A_1,A_2$, and take two fixed elements $0<a \in A_1$ and $0<b\in A_2$.  Then for the strictly positive elements $(1a,4b), (3a,7b), (2a,3b), (2a,8b)$, we have $(1a,4b)+(3a,7b)= (2a,3b)+ (2a,8b)$. Similarly, as in the proof of Lemma \ref{le:RxR}, take two elements $x,y \in G$ such that $x+y\ne y+x$.  Define positive elements
$((a^i)_{i\in I},x)$, $((b^i)_{i\in I},y)$, $((c^i)_{i\in I},x)$ and $((d^i)_{i\in I},y)$ in $(A\lex G)^+$ such that $a^1=1a, a^2=4b$, $b^1 =3a, b^2=7b$, $c^1=2a, c^2=3b$, and  $d^1=2a, d^2=8b$.

Now we show that any RDP table

\[
\begin{matrix}
((a^i)_{i \in I},x) &\vline & ((n_{11}^i)_{i \in I},c_{11}) & ((n_{12}^i)_{i\in I},c_{12})\\
((b^i)_{i \in I},y) &\vline & ((n_{21}^i)_{i \in I},c_{21}) & ((n_{22}^i)_{i \in I},c_{22})\\
  \hline     &\vline      &((c^i)_{i \in I},x) & ((d^i,y)
\end{matrix}\ \  (B)
\]
for $((a^i)_{i\in I},x)+((b^i)_{i\in I},y)=((c^i)_{i\in I},x)+((d^i)_{i\in I},y)$ gives no RDP$_1$ table. In fact, we have two kinds of equations
$$((1a,4b),x)+ ((3a,7b),y)= ((2a,3b)x) +((2a,8b),y)$$
and
$$ a^i+b^i=c^i+d^i$$
for each $i \in I\setminus \{1,2\}$.

For the first one we have from (B) the following RDP table

$$
\begin{matrix}
((1a,4b),x)  &\vline & ((n^1_{11},n^2_{11},c_{11}) & ((n^1_{12},n^2_{12}),c_{12})\\
((3a,7b),y) &\vline & ((n^1_{21},n^2_{21}),c_{21}) & ((n^1_{22},n^2_{22}),c_{22})\\
  \hline     &\vline      &((2a,3b),x) & ((2a,8b),y)
\end{matrix}\ \ .
$$

The elements $(3a,7b)$ and $(2a,3b)$ are non-comparable as well as are $(1a,4b)$ and $(2a,8b)$. Therefore, $n^1_{12},n^2_{12}, n^1_{21},n^2_{21}>0$. There are non-zero elements $s\in A_1$ and $t\in A_2$ such that $0<s< n^1_{12},n^1_{21}$ and $0<t< n^2_{12},n^2_{21}$. Hence $0\le ((s,t),x)\le (n^1_{12},n^2_{12}),c_{12})$ and $((s,t),x)\le (n^1_{21},n^2_{21}),c_{21})$ but $((s,t),x)+((s,t),y) \ne ((s,t),y)+((s,t),x)$.

Therefore, RDP$_1$ fails for $A\lex G$.
\end{proof}

\begin{thm}\label{th:RDP10}
	Let $\{A_i\mid i\in I\}$ be a family of linearly ordered po-groups and $G$ be a directed Abelian po-group with $\RDP_1$.
	Then $(\prod_{i\in I}A_i)\lex G$ has $\RDP_1$.
\end{thm}

\begin{proof}
Let elements $((a^i)_{i \in I},x),((b^i)_{i \in I},y),((c^i)_{i \in I},z),((d^i)_{i \in I},u)\in (\prod_{i\in I}A_i \lex G)^+$ satisfy
$((a^i)_{i \in I},x)+((b^i)_{i \in I},y)=((c^i)_{i \in I},z)+((d^i)_{i \in I},u)$.
 Then clearly, $(a^i)_{i \in I},(b^i)_{i \in I},(c^i)_{i \in I},(d^i)_{i \in I}\geq (0)_{i\in I}$. Since every $A_i$ is linearly ordered, $A_i$ is an $\ell$-group, so it satisfies RDP$_2$ and RDP$_1$.

(1) If $(a^i)_{i \in I}< (c^i)_{i \in I}$, then for each $i\in I$, $a^i\leq c^i$ and there exists $j\in I$ such that $a_j<c_j$
and so we have the following RDP$_1$ table:
\[
\begin{matrix}
((c^i)_{i \in I},z) &\vline & ((a^i)_{i \in I},x) & ((-a^i+c^i)_{i \in I},-x+z)\\
((d^i)_{i \in I},u) &\vline & ((0)_{i \in I},0) & ((d^i)_{i \in I},u)\\
  \hline     &\vline      &((a^i)_{i \in I},x) & ((b^i)_{i \in I},y)
\end{matrix}\ \ .
\]

(2) If $(a^i)_{i \in I}> (c^i)_{i \in I}$, then similarly to (1) we can find an RDP$_1$  table.

(3) If $(a^i)_{i \in I}= (c^i)_{i \in I}$, then clearly $(b^i)_{i \in I}=(d^i)_{i \in I}$.

(i) If $(a^i)_{i \in I}=(0)_{i\in I}$ and $(b_{i})_{i\in I}=(0)_{i\in I}$, then clearly we can find an RDP$_1$  table.

(ii) If $(a^i)_{i \in I}=(0)_{i\in I}$ and $(b_{i})_{i\in I}\neq (0)_{i\in I}$, then we have $x,z\geq 0$. Let $t$ be a lower bound for
$y,u$. Consider an RDP$_1$  decomposition for
$x+(y-t)=z+(u-t)$ as follows:
\[
\begin{matrix}
z &\vline & c_{11} & c_{12}\\
u-t &\vline & c_{21} & c_{22}\\
  \hline     &\vline      & x & y-t
\end{matrix}\ \ .
\]
Then we have
\[
\begin{matrix}
z &\vline & c_{11} & c_{12}\\
u &\vline & c_{21} & c_{22}+t\\
  \hline     &\vline      & x  & y
\end{matrix}\ \  \quad (*)
\]
and so
\[
\begin{matrix}
((c^i)_{i \in I},z) &\vline & ((0)_{i\in I},c_{11}) & ((0)_{i\in I},c_{12})\\
((d^i)_{i \in I},u) &\vline & ((0)_{i\in I},c_{21}) & ((b^i)_{i \in I},c_{22}+t)\\
  \hline     &\vline      &((a^i)_{i \in I},x) & ((b^i)_{i \in I},y)
\end{matrix}\ \
\]
is an RDP$_1$  table for $((a^i)_{i \in I},x)+((b^i)_{i \in I},y)=((c^i)_{i \in I},z)+((d^i)_{i \in I},u)$.

(iii)  If $(a^i)_{i \in I}\neq (0)_{i\in I}$ and $(b_{i})_{i\in I}= (0)_{i\in I}$,  then similarly to (iii) we can find an RDP$_1$  table for
$((a^i)_{i \in I},x)+((b^i)_{i \in I},y)=((c^i)_{i \in I},z)+((d^i)_{i \in I},u)$.

(iv) Let $(a^i)_{i \in I}\neq (0)_{i\in I}$ and $(b_{i})_{i\in I}\neq (0)_{i\in I}$. Since $G$ is directed, there exists $t\in G$ such that $t\leq x,y,z,u$.  Let
\[
\begin{matrix}
(-t+z) &\vline & c_{11} & c_{12}\\
(u-t) &\vline & c_{21} & c_{22}\\
  \hline     &\vline      &(-t+x) & (y-t)
\end{matrix}\ \
\]
be an RDP$_1$  table for $(-t+z)+(u-t)=(-t+x)+(y-t)$. Then we have
\[
\begin{matrix}
z &\vline & t+c_{11} & c_{12}\\
u &\vline & c_{21} & c_{22}+t\\
  \hline     &\vline      & x & y
\end{matrix}\ \ . \quad (**)
\]
It gives an RDP$_1$  table for $((a^i)_{i \in I},x)+((b^i)_{i \in I},y)=((c^i)_{i \in I},z)+((d^i)_{i \in I},u)$
\[
\begin{matrix}
((c^i)_{i \in I},z) &\vline & ((a^i)_{i \in I},t+c_{11}) & ((0)_{i\in I},c_{12})\\
((d^i)_{i \in I},u) &\vline & ((0)_{i\in I},c_{21}) & ((d^i)_{i \in I},c_{22}+t)\\
  \hline     &\vline      &((a^i)_{i \in I},x) & ((b^i)_{i \in I},y)
\end{matrix}\ \ .
\]

(4) Let $(c^i)_{i \in I}$ and $(a^i)_{i \in I}$ be not comparable. Consider the following subsets of $I$
	$$I_1:=\{i\in I|\ a^i<c^i\},\quad\!\! I_2:=\{i\in I|\ c^i<a^i\},\quad\!\! I_3:=\{i\in I|\ a^i=c^i\}.
	$$
	By the assumption, $I_1\neq\emptyset$ and $I_2\neq\emptyset$, and if $i \in I_3$, then $d^i=b^i$.
	Set
	
	\begin{equation}
		e^i= \left\{\begin{array}{lllll}
			a^i & \mbox{ if }  i\in I_1\\
			c^i & \mbox{ if } i\in I_2\\
			a^i& \mbox{ if } i\in I_3
		\end{array}\right.
		\hspace{1cm}
		f^i= \left\{\begin{array}{llll}
			-a^i+c^i  & \mbox{ if } i\in I_1\\
			0 & \mbox{ if } i\in I_2 \\
			0 &  \mbox{ if } i\in I_3
		\end{array}\right.
	\end{equation}
	% % % % % % % % % % % % % % % %
	\begin{equation}
		g^i= \left\{\begin{array}{lllll}
			0 & \mbox{ if } i\in I_1\\
			-c^i+a^i & \mbox{ if } i\in I_2\\
			0& \mbox{ if } i\in I_3
		\end{array}\right.
		\hspace{1cm}
		h^i= \left\{\begin{array}{llll}
			d^i  & \mbox{ if } i\in I_1\\
			b^i & \mbox{ if } i\in I_2 \\
			d^i &  \mbox{ if } i\in I_3.
		\end{array}\right.
	\end{equation}
	Since $G$ is directed, there exists a $d\in G$ such that $d\leq x,y,z,u$.
	Let
		\[
		\begin{matrix}
		-d+z &\vline & c_{11} & c_{12}\\
		u-d &\vline & c_{21} & c_{22}\\
		\hline     &\vline      &-d+x & y-d
		\end{matrix}\ \
		\]
be an \RDP$_1$ table for $(-d+x)+(y-d)=(-d+z)+(u-d)$. Since $G$ is Abelian,  we have
			\[
			\begin{matrix}
			z &\vline & c_{11} & c_{12}+d\\
			u &\vline & c_{21}+d & c_{22}\\
			\hline     &\vline      & x & y
			\end{matrix}\ \ .
			\]
Then
	\[
	\begin{matrix}
	((c^i)_{i \in I},z) &\vline & ((e^i)_{i \in I},c_{11}) & ((f^i)_{i \in I},c_{12}+d)\\
	((d^i)_{i \in I},u) &\vline & ((g^i)_{i \in I},c_{21}+d) & ((h^i)_{i \in I},c_{22})\\
	\hline     &\vline      &((a^i)_{i \in I},x) & ((b^i)_{i \in I},y)
	\end{matrix}\ \
	\]
	is an RDP$_1$ table for $((a^i)_{i \in I},x)+((b^i)_{i \in I},y)=((c^i)_{i \in I},z)+((d^i)_{i \in I},u)$. Indeed,
	let $((0)_{i\in I},0)\leq ((k^i)_{i\in I},w)\leq  ((f^i)_{i \in I},c_{12}+d)$ and
	$((0)_{i\in I},0)\leq ((m^i)_{i\in I},v)\leq  ((g^i)_{i \in I},c_{21}+d)$. Then
	$(0)_{i\in I}\leq (k^i)_{i\in I}\leq (f^i)_{i \in I}$ and $(0)_{i\in I}\leq (m^i)_{i\in I}\leq (g^i)_{i \in I}$.
	If $i\in I-I_1$, then $f^i=0$ and so $k^i=0$. That is, $k^i=0$  for all $i\in I-I_1$.  Similarly,
	$m^i=0$ for all $i\in I-I_2$. Put $i\in I$. If $m^i\neq 0$, then $i\in I_2$, so $k^i=0$. It follows that $m^i+k^i=k^i+m^i$.
	For $m^i=0$ clearly, $m^i+k^i=k^i+m^i$. Thus $(m^i)_{i\in I}+(k^i)_{i\in I}=(k^i)_{i\in I}+(m^i)_{i\in I}$. Since $G$ is Abelian,
	then we have $((m^i)_{i\in I},v)+((k^i)_{i\in I},w)=((k^i)_{i\in I},w)+((m^i)_{i\in I},v)$. 	
	
	From (1)--(4) we get that $(\prod_{i\in I}A_i)\lex G$ has $\RDP_1$.
\end{proof}

\section{Com-Directness Property and Lexicographic Product}%4

In this section we show that if we put more general conditions to the second factor, $G$, of the lexicographic product $A\lex G$, than $G$ is Abelian, we can extend the class of po-groups  with RDP such that $A\lex G$ has RDP for each po-group $A$ with RDP. Such a condition is the com-directness of $G$.

For the aims of the following theorem, we introduce a stronger form of the directness of a po-group which is motivated as follows. If $G$ is a po-group, then for each $x,y\in G^+$, there is $d \in G$ such that $d\le x,y$  and $d$ commutes with $x$ and with $y$; in this case such $d$ can be trivially used $d=0$. The same is true for a directed Abelian po-group. Therefore, we say that a po-group $G$ is {\it com-directed} (com stands for the commutativity), if given $x,y\in G$, there is a $d\in \mbox{Z}(G)$, where $\mbox{Z}(G)$ is the center of $G$, such that $d\le x,y$.  This is equivalent, given $x,y \in G$, there is a $d\in \mbox{Z}(G)$ such that $x,y \le d$. Of course, if $G$ is com-directed, then $G$ is directed, too. If $G$ is Abelian, then both notions, directness and com-directness, coincide.

We note that if a po-group $G$ is com-directed, then it does not follow that $G$ is Abelian. Indeed, let $G$ be a po-group that is not Abelian. Then $\mathbb Z \lex G$ is a com-directed po-group that is not Abelian; indeed, given $(n,g),(m,h) \in \mathbb Z \lex G$,  any element $d=(k,0)\in \mathbb Z\lex G$, where $k <n,m$, is from the center $\mbox{Z}(\mathbb Z\lex G)$, and we have $d<(n,g),(m,h)$.

\begin{thm}\label{th:RDP11}
Let $(A;+,0)$ be a po-group and $(G;+,0)$ be a com-directed po-group. Then $A\lex G$ satisfies {\rm RDP} if and only if both $A$ and $G$ satisfy \RDP.
\end{thm}

\begin{proof}
Let $A\lex G$ satisfy RDP, and let $x_1+x_2 =y_1+y_2$ in $A$ for $x_1,x_2,y_1,y_2 \in A^+$. Then $(x_1,0)+(x_2,0)=(y_1,0)+(y_2,0)$ which easily implies that $x_1+x_2 =y_1+y_2$ has an RDP table in $A$. If now for $u_1,u_2,v_1,v_2 \in G^+$ we have $u_1+u_2=v_1+v_2$, then $(0,u_1)+(0,u_2)=(0,v_1)+(0,v_2)$. The RDP in $A\lex G$ implies that $u_1+u_2=v_1+v_2$ has an RDP table in $G$.

Now let $A$ and $G$ have RDP. If $G$ is the trivial po-group, i.e. $G=\{0\}$, then $A\lex G =A\lex \{0\} \cong A$ and $A\lex G$ satisfies RDP.

Let us assume that $G$ is a non-trivial com-directed po-group with RDP. Let us have elements $(x_1,u_1), (x_2,u_2),(y_1,v_1), (y_2,v_2) \in (A\lex G)^+$ such that
$$
(x_1,u_1)+ (x_2,u_2)=(y_1,v_1) + (y_2,v_2). \eqno(4.1)
$$

(I) First we assume that $x_1$ and $y_1$ are comparable. In view of (4.1), $x_2$ and $y_2$ are also comparable. We have the following 9 subcases.

(i) Let $(0,u_1)+(0,u_2) = (0,v_1)+ (0,v_2).$  Then $u_1,u_2,v_1,v_2 \in G^+$ and RDP for this case  follows from RDP for $G.$

(ii) $(0,u_1) + (x,u_2 )=
(0,  v_1 ) + (y,  v_2 )$ for $u_1,v_1 \ge 0$, $u_2,v_2 \in G$ for each  $x=y \in A^+\setminus\{0\}$. Then $u_1+u_2=v_1+v_2$. While $G$ is directed,  there is an element $d \in G$ such that $u_2,v_2 \ge d$. Then $u_1+(u_2-d)=v_1+(v_2-d)$ and for them we have an RDP decomposition

$$
\begin{matrix}
u_1 &\vline & c_{11} & c_{12}\\
u_2-d &\vline & c_{21} & c_{22}\\
  \hline     &\vline      &v_1 & v_2-d
\end{matrix}\ \ .
$$
 Then

$$
\begin{matrix}
u_1  &\vline & c_{11} & c_{12}\\
u_2 &\vline & c_{21} & c_{22}+d\\
  \hline     &\vline      &v_1 & v_2
\end{matrix}\ \
$$
and
$$
\begin{matrix}
(0, u_1)  &\vline & (0, c_{11}) & (0, c_{12})\\
(x, u_2) &\vline & (0, c_{21}) & (y, c_{22}+d)\\
  \hline     &\vline      &(0, v_1) & (y, v_2)
\end{matrix}\ \
$$
is an RDP decomposition for (ii) in the po-group $A\lex G$.

(iii) $(x,   u_1 ) + (0,  u_2 )=
(y,  v_1 ) + (0,  v_2 )$ for $u_2,v_2 \ge 0$, $u_1,v_1 \in G$ for  $x=y \in A^+\setminus \{0\}$. The directness of $G$ implies, there is $d \in G$ such that $d\le u_1,u_2,v_1,v_2$. Equality (iii) can be rewritten in the equivalent form $(x,   -d+u_1 ) + (0,  u_2-d )=
(y,  -d+v_1 ) + (0,  v_2-d )$ which yields $(-d+u_1)+(u_2-d)=(-d+v_1)+(v_2-d)$.
It entails an RDP decomposition in the po-group $G$

$$
\begin{matrix}
-d+u_1 &\vline & c_{11} & c_{12}\\
u_2-d &\vline & c_{21} & c_{22}\\
  \hline     &\vline      &-d+v_1 & v_2-d
\end{matrix}\ \ ,
$$
consequently,
$$
\begin{matrix}
u_1  &\vline & d+c_{11} & c_{12}\\
u_2 &\vline & c_{21} & c_{22}+d\\
  \hline     &\vline      &v_1 & v_2
\end{matrix}\ \ ,
$$
and it gives an RDP decomposition of (iii) in the po-group $A\lex G$

$$
\begin{matrix}
(x, u_1)  &\vline & (x, d+c_{11}) & (0, c_{12})\\
(0, u_2) &\vline & (0, c_{21}) & (0, c_{22}+d)\\
  \hline     &\vline      &(y, v_1) & (0, v_2 )
\end{matrix}\ \ .
$$

(iv) $(x,   u_1 ) + (0,  u_2 )=
(0,  v_1) + (y,  v_2 )$ for $u_1,v_2\in G $, $u_2,v_1 \ge 0$ for  $x=y \in A^+\setminus\{0\}$.

Then $u_1+u_2 = v_1+v_2$  which implies $-v_1+u_1= v_2- u_2$.   If we use the decomposition

$$
\begin{matrix}
(x, u_1)  &\vline & (0, v_1) & (x, -v_1+u_1)\\
(0, u_2) &\vline & (0,0) & (0, u_2)\\
  \hline     &\vline      &(0, v_1) & (y, v_2 )
\end{matrix}\ \ ,
$$
we see that it gives an RDP decomposition for (iv).

(v) $(x,   u_1) + (0,  u_2 )=
(y_1,  v_1 ) + (y_2,  v_2 )$ for $u_1,v_1, v_2\in G $, $u_2\ge 0$, where $y_1,y_2\in A^+\setminus\{0\}$ and $y_1+y_2=x$. Then $u_1+u_2=v_1+v_2$. Hence, the following table gives an RDP decomposition for (v)

$$
\begin{matrix}
(x, u_1)  &\vline & (y_1, v_1) & (y_2, -v_1+u_1)\\
(0, u_2) &\vline & (0,0) & (0, u_2)\\
  \hline     &\vline      &(y_1, v_1) & (y_2, v_2)
\end{matrix}\ \ .
$$

(vi) $(0,   u_1) + (x,  u_2)=
(y_1,  v_1) + (y_2,  v_2 )$ for $u_2,v_1,v_2\in G $, $u_1\ge 0$, where $y_1,y_2\in A^+\setminus\{0\}$ and $y_1+y_2=x$. Then we have $v_1+v_2=u_1+u_2$ and  the following RDP decomposition

$$
\begin{matrix}
(0, u_1)  &\vline & (0, u_1) & (0,0)\\
(x, u_2) &\vline & (y_1, -u_1+v_1) & (y_2, v_2)\\
  \hline     &\vline      &(y_1, v_1) & (y_2, v_2)
\end{matrix}\ \ .
$$

(vii) $(x_1,   u_1 ) + (x_2,  u_2 )=
(y_1,  v_1) + (y_2,  v_2 )$ for $u_1,u_2,v_1,v_2\in G$, where $ x_1,x_2, y_1, y_2 \in A^+\setminus\{0\}$, $x_1+x_2=y=y_1+y_2$ and $x_1> y_1$. Then $u_1+u_2=v_1+v_2$, and since $-y_1+y=y_2$, (vii)  has the following RDP decomposition

$$
\begin{matrix}
(x_1, u_1)  &\vline & (y_1, v_1) & (-y_1+x_1, -v_1+u_1)\\
(x_2, u_2 ) &\vline & (0,0) & (x_2, u_2)\\
  \hline     &\vline      &(y_1, v_1) & (y_2, v_2)
\end{matrix}\ \ \mbox{ if } x_1 >y_1
$$
is an RDP decomposition.

(viii) $(x_1,   u_1) + (x_2,  u_2 )=
(y_1,  v_1 ) + (y_2,  v_2 )$ for $u_1,u_2,v_1,v_2 \in G $,  where $x_1,x_2, y_1, y_2 \in A^+\setminus\{0\}$, $x_1+x_2=y=y_1+y_2$ and $y_1> x_1$. Then
(viii) follows from (vii) when we rewrite (viii) in the equivalent form $(y_1,  v_1 ) + (y_2,  v_2 )=(x_1,u_1) + (x_2,u_2 )$, and an RDP table is as follows

$$
\begin{matrix}
(x_1, u_1)  &\vline & (y_1, v_1) & (-y_1+x_1, -v_1+u_1)\\
(x_2, u_2) &\vline & (0,0) & (x_2, u_2)\\
  \hline     &\vline      &(y_1, v_1) & (y_2, v_2)
\end{matrix}\ \ \mbox{ if } y_1 >x_1.
$$

(ix) $(x_1,   u_1 ) + (x_2,  u_2)=
(y_1,  v_1) + (y_2,  v_2 )$ for $u_1,u_2,v_1,v_2\in G_2 $, where $x_1, x_2, y_1,$ $ y_2 \in A^+\setminus\{0\}$, $x_1+x_2=x=y_1+y_2$ and $x_1= y_1$.  Then $u_1+u_2= v_1+v_2$. The directness of $G$ entails that there is $d\in G$ such that $u_1,u_2,v_1,v_2\ge d$. Hence, $(-d+u_1)+(u_2-d)= (-d+v_1)+(v_2-d)$. The RDP holding in $G$ entails the following RDP table

$$
\begin{matrix}
-d+u_1  &\vline & c_{11} & c_{12}\\
u_2-d &\vline & c_{21} & c_{22}\\
  \hline     &\vline      &-d+v_1 & v_2-d
\end{matrix}\ \ ,
$$
so that

$$
\begin{matrix}
u_1  &\vline & d+c_{11} & c_{12}\\
u_2 &\vline & c_{21} & c_{22}+d\\
  \hline     &\vline      &v_1 & v_2
\end{matrix}\ \ .
$$
It gives an RDP decomposition of (ix)

$$
\begin{matrix}
(x_1, u_1 )  &\vline & (x_1, d+c_{11}) & (0, c_{12} )\\
(x_2,  u_2) &\vline & (0, c_{21}) & (x_2, c_{22}+d)\\
  \hline     &\vline      &(x_1, v_1)& (x_2, v_2 )
\end{matrix}\ \ .
$$

(II) Let $x_1$ and $y_1$ be not comparable. In particular, we have $x_1>0$ and $y_1>0$. Since in $A$ we have the RDP property, from $x_1+x_2=y_1+y_2$ we have an RDP table

$$
\begin{matrix}
x_1 &\vline & e_{11} & e_{12}\\
x_2 &\vline & e_{21} & e_{22}\\
  \hline     &\vline      &y_1 & y_2
\end{matrix}\ \ .
$$

By the assumptions, there is $d \in \mbox{Z}(G)$ such that $d\le u_1,u_2, v_1,v_2$. Hence, for the equality of positive elements $(-d+u_1)+(u_2-d)=(-d+v_1)+(v_2-d)$ in $G$, there is an RDP table

$$
\begin{matrix}
-d+u_1 &\vline & c_{11} & c_{12}\\
u_2-d &\vline & c_{21} & c_{22}\\
  \hline     &\vline      &-d+v_1 & v_2-d
\end{matrix}\ \ ,
$$
which entails

$$
\begin{matrix}
u_1 &\vline & d+c_{11} & c_{12}\\
u_2 &\vline & c_{21} & c_{22}+d\\
  \hline     &\vline      &v_1 & v_2
\end{matrix}\ \ .
$$

If $e_{11}>0$ and $e_{22} >0$, then

$$
\begin{matrix}
(x_1,u_1) &\vline & (e_{11}, d+c_{11}) & (e_{12},c_{12})\\
(x_2,u_2) &\vline & (e_{21},c_{21}) & (e_{22},c_{22}+d)\\
  \hline     &\vline      &(y_1,v_1) & (y_2,v_2)
\end{matrix}\ \
$$
is an RDP table for (4.1).

If $e_{11}=0$, then $e_{12}>0$ and $e_{21}>0$ (otherwise, $x_1$ and $y_1$ are comparable). Then the following table gives an RDP decomposition for (4.1)

$$
\begin{matrix}
(x_1,u_1) &\vline & (e_{11},c_{11}) & (e_{12},d+c_{12})\\
(x_2,u_2) &\vline & (e_{21},c_{21}+d) & (e_{22},c_{22})\\
  \hline     &\vline      &(y_1,v_1) & (y_2,v_2)
\end{matrix}\ \ .
$$

If $e_{22}=0$, then $e_{12}>0$ and $e_{21}>0$, and the last RDP table is an RDP table also for this case.

Summing up all above cases, we see that $A\lex G$ satisfies RDP.
\end{proof}

\begin{rmk}
(i) We note that the com-directness property of $G$ was used only in the last table, and in all the other cases, we have used only the assumption $G$ is directed. Therefore, there is a natural question, does Theorem \ref{th:RDP11} hold assuming $G$ is not necessarily com-directed and rather directed?

(ii) If $G$ is a directed Abelian po-group with RDP, then $G$ is com-directed and $A\lex G$ has RDP for each po-group $A$ with RDP.

(iii) Theorem \ref{th:RDP11} does not hold in the RDP$_1$ variant, in general. Indeed, let $G=\mathbb Z \lex H$, where $H$ is a directed non-Abelian po-group with RDP$_1$. By the note just before Theorem \ref{th:RDP11}, it was shown that $G$ is a com-directed non-Abelian po-group, and applying \cite[Thm 3.3]{262}, we get $G$ has RDP$_1$. Using Lemma \ref{le:RxR}, we see that if $A = \mathbb R \times \mathbb R$ is with strict product ordering, then the Abelian po-group $A$ has both RDP and RDP$_1$, but $A \lex G$ has only RDP and RDP$_1$ fails in it.
\end{rmk}

\begin{rmk}\label{re:RDP12}

A partial answer for the latter note (i) is the assumption that $A$ satisfies the following condition:

Given $a,b \in A^+\setminus\{0\}$ such that $a$ and $b$ are not comparable, there is $d' \in A$, $0<d' \le a,b$ such that $a+d'=d'+a$ and $b+d'=d'+b$.

Hence, if $G$ is directed, and $A$ and $G$ satisfy RDP, then $A\lex G$ has RDP.
\end{rmk}

\begin{proof}
It is enough to verify the last case of (II) in the proof of Theorem \ref{th:RDP11}.

Since $e_{12}>0$ and $e_{21}>0$ are not comparable, there is $0<d'\in A$ such that $d'\le e_{12},e_{21}$ and $d'$ commutes with $e_{21}$ and $e_{12}$.
Then the following table

$$
\begin{matrix}

(x_1,u_1) &\vline & (e_{11}+d', d+c_{11}) & (-d'+e_{12},c_{12})\\
(x_2,u_2) &\vline & (-d'+e_{21},c_{21}) & (d'+e_{22},c_{22}+d)\\
  \hline     &\vline      &(y_1,v_1) & (y_2,v_2)
\end{matrix}\ \
$$
is an RDP table for the last case of (II) in the proof of Theorem \ref{th:RDP11}.
\end{proof}

Remark \ref{re:RDP12} can be generalized as follows.

\begin{rmk}\label{re:RDP13}
Let $A$ be a po-group with {\rm RDP} such that, given two non-comparable elements $a,b \in A^+\setminus\{0\}$, there is an element $d' \in A$ such that $0<d'\le a,b$ and $-a+d'+a=-b+d'+b$. If $G$ satisfies {\rm RDP}, then $A\lex G$ has {\rm RDP}.
\end{rmk}

\begin{proof}
It is enough to verify the last case of (II) in the proof of Theorem \ref{th:RDP11}.

Since $e_{12}>0$ and $e_{21}>0$ are not comparable, there is $0<d'\in A$ such that $d'\le e_{12},e_{21}$ and $-e_{12}+d'+e_{12}= -e_{21}+d'+e_{21}$.
Then the following table

$$
\begin{matrix}

(x_1,u_1) &\vline & (e_{11}+d', d+c_{11}) & (-d'+e_{12},c_{12})\\
(x_2,u_2) &\vline & (-d'+e_{21},c_{21}) & (-e_{12}+d'+e_{12}+e_{22},c_{22}+d)\\
  \hline     &\vline      &(y_1,v_1) & (y_2,v_2)
\end{matrix}\ \
$$
is an RDP table for the last case of (II) in the proof of Theorem \ref{th:RDP11}. Indeed, we have
$-d'+e_{21}+(-e_{12}+d'+e_{12}+e_{22})=-d'+e_{21}+(-e_{21}+d'+ e_{21}+ e_{22})= e_{21}+e_{22}=x_2$.
\end{proof}

\section{Non-Comparability Directness Property and Lexicographic Product}%5

We continue with study of the Riesz Decomposition Properties of the lexicographic product of po-groups. In particular, we introduce the Non-Comparability Directness Property, and some illustrating examples will be present.

A po-group $A$ satisfying the condition ``given two non-comparable elements $a,b \in A^+\setminus\{0\}$, there is an element $d \in A$ such that $0<d\le a,b$ and $-a+d+a=-b+d+b$" is said to be a po-group with (or satisfying) the {\it Non-Comparability Directness Property}, NCDP for short.

\begin{thm}\label{th:RDP13'}
{\rm (i)}  Let $A$ be a po-group with {\rm RDP} such that, given two non-comparable elements $a,b \in A^+\setminus\{0\}$ with $a+b=b+a$,
there is an element $d' \in A$ such that $0<d'\le a,b$ and $-a+d'+a=-b+d'+b$. If a po-group $G$ satisfies {\rm RDP}, then $A\lex G$ has {\rm RDP}.

{\rm (ii)} Let $A$ be a po-group with {\rm RDP} and $G$ be a directed po-group.
Let $A$ do not satisfy the condition in part {\rm (i)} and let  $A\lex G$ have  {\rm RDP}, then $G$ has {\rm RDP} and,
for each equation $u_1+u_2=v_1+v_2$ of elements of $G$, there exist $d_1,d_2\in G$ such that
\begin{itemize}
\item[{\rm (1)}] $d_1\leq u_1,v_2$ and $d_2\leq u_2,v_1$;
\item[{\rm (2)}]  $d_1+d_2=d_2+d_1$;
\item[{\rm (3)}]  $-u_1+v_1=-d_1+d_2$.
\end{itemize}

{\rm (iii)} Let $G$ have {\rm RDP}. If, for each equation $u_1+u_2=v_1+v_2$ of elements of a po-group $G$, there exist $d_1,d_2\in G$ such that conditions {\rm (1)}--{\rm (3)} of {\rm (ii)} are satisfied, then $G$ is directed and $A \lex G$ has {\rm RDP} for each po-group $A$ with {\rm RDP}.
\end{thm}

\begin{proof}
(i)  It is enough to verify the last case of (II) in the proof of Theorem \ref{th:RDP11}.
From $x_1+x_2=y_1+y_2$ it follows that
$e_{11}+e_{12}+e_{21}+e_{22}=e_{11}+e_{21}+e_{12}+e_{22}$ and so
 $e_{12}+e_{21}=e_{21}+e_{12}$.
Since $e_{12}>0$ and $e_{21}>0$ are not comparable, there is $0<d'\in A$ such that $d'\le e_{12},e_{21}$ and $-e_{12}+d'+e_{12}= -e_{21}+d'+e_{21}$.
Then the following RDP table

$$
\begin{matrix}

(x_1,u_1) &\vline & (e_{11}+d', d+c_{11}) & (-d'+e_{12},c_{12})\\
(x_2,u_2) &\vline & (-d'+e_{21},c_{21}) & (-e_{12}+d'+e_{12}+e_{22},c_{22}+d)\\
  \hline     &\vline      &(y_1,v_1) & (y_2,v_2)
\end{matrix}\ \
$$
is for the last case of (II) in the proof of Theorem \ref{th:RDP11}. Indeed, we have
$-d'+e_{21}+(-e_{12}+d'+e_{12}+e_{22})=-d'+e_{21}+(-e_{21}+d'+ e_{21}+ e_{22})= e_{21}+e_{22}=x_2$.

(ii)  By the assumption there are non-comparable $a,b\in A^{+}\setminus \{0\}$ with $a+b=b+a$ such that  there is no $0<d\leq a,b$
satisfying the condition $-a+d-a=-b+d+b$.  First we assume that $A\lex G$ has RDP. Then clearly $G$ has RDP.
Choose arbitrary elements $u_1,u_2,v_1,v_2\in G$ such that $u_1+u_2=v_1+v_2$. Then $(a,u_1),(b,u_2),(b,v_1),(a,v_2)$ are positive elements of
$A\lex G$ and
$(a,u_1)+(b,u_2)=(b,v_1)+(a,v_2)$
and so we have an RDP table in $A\lex G$ as follows

$$
\begin{matrix}
(a,u_1) &\vline & (e_{11},c_{11}) & (e_{12},c_{12})\\
(b,u_2) &\vline & (e_{21},c_{21}) & (e_{22},c_{22})\\
  \hline     &\vline      &(b,v_1) & (a,v_2)
\end{matrix}\ \ .
$$

Clearly, $0\leq e_{ij}$ for all $i,j\in\{1,2\}$. If $e_{11}>0$, then $e_{12}<a$ (since $a=e_{11}+e_{12}$).
%Similarly, $e_{21}<b$.
Also, $-a+e_{11}+a=-e_{12}-e_{11}+e_{11}+e_{12}+e_{22}=e_{22}$ and
$-b+e_{11}+b=-e_{21}-e_{11}+e_{11}+e_{21}+e_{22}=e_{22}$, so
$-a+e_{11}+a=-b+e_{11}+b$ which is a contradiction. Therefore, $e_{11}=0$. In a similar way, $e_{22}=0$. Thus $c_{11}$ and $c_{22}$ are positive elements of $G$. From $u_1=c_{11}+c_{12}$ and $v_{2}=c_{12}+c_{22}$, we get that $c_{12}\leq u_1,v_2$. Similarly,
$c_{21}\leq u_2,v_1$. Clearly, $c_{12}+c_{21}=c_{21}+c_{12}$. Moreover,
$-u_1+v_1=-c_{12}-c_{11}+c_{11}+c_{21}=-c_{12}+c_{21}$. Set $d_1=e_{12}$ and $d_2=e_{21}$.

(iii) Suppose that $G$ has RDP and, for each equation $u_1+u_2=v_1+v_2$ of elements of $G$, there are $d_1,d_2\in G$ satisfying the above mentioned properties (1)--(3). First we show that $G$ is directed. Indeed, let $u,v\in G$ be given. For the equality $u +(u+v)=2u +v$, there is $d_1\in G$ such that $d_1\le u, v$.

Now let $A$ be a po-group with RDP. We show that $A\lex G$ has RDP.
It is enough to verify the last case of (II) in the proof of Theorem \ref{th:RDP11}.
By the assumption, there are $d_1,d_2\in G$ such that $d_1\leq u_1,v_2$, $d_2\leq u_2,v_1$, $d_1+d_2=d_2+d_1$ and $-u_1+v_1=-d_1+d_2$.

If $e_{11}=0$ or $e_{22}=0$, then $e_{12}>0$ and $e_{21}>0$, we claim that we have an RDP table
$$
\begin{matrix}

(x_1,u_1) &\vline & (e_{11}, u_1-d_1) & (e_{12},d_1)\\
(x_2,u_2) &\vline & (e_{21},d_2) & (e_{22},-d_2+u_2)\\
  \hline     &\vline      &(y_1,v_1) & (y_2,v_2)
\end{matrix}\ \
$$
for $(x_1,u_1)+(x_2,u_2)=(y_1,v_1)+(y_2,v_2)$. Clearly,  $(e_{11}, u_1-d_1), (e_{12},d_1), (e_{21},d_2),
(e_{22},-d_2+u_2)$ are positive elements of $A\lex G$. Also,
$d_1-d_2+u_2=-(-u_1+v_1)+u_2=-v_1+u_1+u_2=v_2$ and
$u_1-d_1+d_2=u_1-u_1+v_1=v_1$, so $A \lex G$ has RDP.
\end{proof}

Inspired by (iii) of the latter theorem, we say that a po-group $G$ has {\it wRDP}, (w stands for strong) if,  for each equation $u_1+u_2=v_1+v_2$ of elements of $G$, there exist $d_1,d_2\in G$ such that conditions {\rm (1)}--{\rm (3)} of Theorem \ref{th:RDP13'}(ii) are satisfied.

Now we present an example of a directed non-Abelian po-group $A$ with NCDP and RDP satisfying the assumptions of Theorem \ref{th:RDP13'}(i).

\begin{exm}\label{ex:RDP14}
Consider the po-group $B:=(\mathbb{R}\times \mathbb{R};+, (0,0))$ with the strict product ordering.
Clearly, $B$ is directed and has RDP. Let $(C;+,0)$ be a
non-Abelian linearly ordered group. By \cite[Thm. 3.1]{Dvu2015}, $A=C\lex B$ has RDP.
We claim that $A$ satisfies the assumptions of Theorem \ref{th:RDP13'}(i).
Put $(a_1,a_2),(b_1,b_2)\in A^{+}\setminus\{0\}$ such that $(a_1,a_2)$ and $(b_1,b_2)$ are non-comparable.
Since $C$ is a chain, then $a_1=b_1$ (otherwise, they are comparable), so $a_2$ and $b_2$ must be non-comparable.
Let $a_2=(x,y)$ and $b_2=(u,v)$.

(i) Let $a_1>0$. Clearly, we can find an element $d\in \mathbb{R}\times\mathbb{R}$ such that $d$ is strictly less than
$a_2$ and $b_2$ and so $(a_1,d)$ is a strictly positive element of $C\lex B$ and $(a_1,d)<(a_1,a_2),(b_1,b_2)$.

(ii) Let $a_1= 0$. Then $a_2$ and $b_2$ are positive elements of $B$. Since $B$ has the strict product ordering, it follows that $x,y,u,v>0$.
We can find $s,t\in \mathbb{R}$ such that $0<s<x,u$ and $0<t<y,v$. Clearly,
$d=(s,t)$ is a strictly positive element of $\mathbb{R}\times \mathbb{R}$ and $d<a_2,b_2$. So,
$(a_1,d)$ is a strictly positive element of $A$ which is strictly less than $(a_1,a_2)$ and $(b_1,b_2)$.

In both cases we have $-(a_1,a_2)+(a_1,d)+(a_1,a_2)=(a_1,d)=-(b_1,b_2)+(a_1,d)+(b_1,b_2)$ which proves the claim.
\end{exm}

The latter example can be strengthened as follows:

\begin{exm}\label{ex:RDP15}
If in Example \ref{ex:RDP14}, we assume that $C$ is a linearly ordered non-trivial group such that $\mbox{Z}(A)=\{0\}$, then $A=C \lex B$ satisfies the conditions of Theorem \ref{th:RDP13'}(i) and is not com-directed.

Indeed, let $(a_1,a_2), (b_1,b_2)$ be two elements of $A$ such that $a_1,b_1<0$. Since $\mbox{Z}(A)=\{0\}\times B$, there is no element $(c_1,c_2)$ from $\mbox{Z}(A)$ such that $(c_1,c_2)\le (a_1,a_2), (b_1,b_2)$.

An example of $C$ satisfying our conditions is the class of square  matrices of the form
$$ A(a,b)=
\left( \begin{array}{cc}
a & b \\
0& 1
\end{array}
\right)
$$
for $a>0$, $b \in (-\infty,\infty)$ with usual multiplication of matrices. It is a non-commutative linearly ordered group with the neutral element $A(1,0)$ and with the positive cone consisting of matrices $A(a,b)$ with $a>1$ or $a=1$ and $b\ge0$. For it we have $\mbox{Z}(C)=\{A(1,0)\}$.
\end{exm}

\begin{exm}\label{ex:GxR}
Let $H$ be a directed po-group with RDP. Then $H\lex \mathbb{R}$ satisfies the conditions in
Theorem \ref{th:RDP13'}(i).
\end{exm}

\begin{proof}
Since $H$ is directed, $A=H\lex \mathbb{R}$ is directed. By  Theorem \ref{th:RDP11}, $H\lex \mathbb{R}$ has RDP.
Let $(a_1,a_2)$ and $(b_1,b_2)$ be strictly positive elements of $H\lex \mathbb{R}$ such that $(a_1,a_2)+(b_1,b_2)=(b_1,b_2)+(a_1,a_2)$.

(1) If $0<a_1,b_1$, then set $d=(0,1)$. Clearly, $d>(0,0)$ and $-(a_1,a_2)+d+(a_1,a_2)=(0,1)=-(b_1,b_2)+d+(b_1,b_2)$.

(2) If $a_1=b_1=0$, then $a_2,b_2>0$. Let $t\in \mathbb{R}$ such that $t<a_2,b_2$. Set $d=(0,t)$. We have
$-(a_1,a_2)+d+(a_1,a_2)=(0,t)=-(b_1,b_2)+d+(b_1,b_2)$.

(3) If $a_1=0$ and $b_1>0$, then $a_2>0$. Set $d=(0,t)$, where $0<t<a_2$. We have $0<d<(a_1,a_2),(b_1,b_2)$ and
$-(a_1,a_2)+d+(a_1,a_2)=(0,t)=-(b_1,b_2)+d+(b_1,b_2)$.

(4) If $a_1>0$ and $b_1=0$, similarly to (3), we have $0<d<(a_1,a_2),(b_1,b_2)$ such that
$-(a_1,a_2)+d+(a_1,a_2)=(0,t)=-(b_1,b_2)+d+(b_1,b_2)$.
\end{proof}

\begin{prop}\label{pr:equal}
A po-group $G$ satisfies has {\rm wRDP} if and only if,
for each equation $u_1+u_2=v_1+v_2$ of elements of $G$, there is a positive element
$k\in G$ such that
\begin{itemize}
\item[{\rm (P1)}] $v_2\leq u_1+k$;
\item[{\rm (P2)}]  $u_2-k$ and $v_2-k$ commute. %That is, $(u_2-k)+(v_2-k)=(v_2-k)+(u_2-k)$.
\end{itemize}
\end{prop}

\begin{proof}
Let $u_1,u_2,v_1,v_2\in G$ such that $u_1+u_2=v_1+v_2$. Suppose (P1) and (P2) hold.  We set
$d_1=v_2-k$ and $d_2=u_2-k$.
Clearly, $d_1\leq v_2$ and $d_2\leq u_2$.
From (P1), it follows that $d_1=v_2-k\leq u_1$, and so $d_1\leq u_1$.
Also, $u_1-(v_2-k)+(u_2-k)=u_1+(u_2-k)-(v_2-k)=(u_1+u_2)-k-(v_2-k)=(v_1+v_2)-k-(v_2-k) =v_1$ and so
$u_1-d_1=v_1-d_2$ and $-u_1+v_1=-d_1+d_2$. Since $0\leq u_1-d_1$, then $0\leq v_1-d_2$. That is, $d_2\leq v_1$.
Therefore, $G$ satisfies  RDP.

Conversely, let $G$ satisfy have RDP.
Set $k:=-d_1+v_2$. Then by (1) of Theorem \ref{th:RDP13'}(ii),
$0\leq k$ and $v_2-k=d_1$. By (2) and (3), $u_2-k=u_2-v_2+d_1=-u_1+v_1+d_1=d_2$ and so by (2),
$u_2-k$ and $v_2-k$ commute. Therefore, (P1) and (P2) hold.
\end{proof}

\begin{prop}\label{pr:RDP16}
Let $G$ be a po-group with {\rm RDP} such that, for each
equation $u_1+u_2=v_1+v_2$ of elements of $G$, there is an element $k\in G^+$ satisfying {\rm (P1)--(P2)} of Proposition {\rm \ref{pr:equal}}, then
we always can find a table
$$
\begin{matrix}

u_1 &\vline & c_{11} & c_{12}\\
u_2 &\vline & c_{21} & c_{22}\\
  \hline     &\vline      &v_1& v_2
\end{matrix}\ \
$$
such that $c_{11},c_{22}\in G^{+}$.
\end{prop}

\begin{proof}
Let a po-group $G$ with RDP satisfying (P1)--(P2) be given. According to Proposition \ref{pr:equal} and Theorem \ref{th:RDP13'}(iii), $A\lex G$ has RDP for each po-group $A$ with RDP, in particular, for the $\ell$-group $A=\mathbb R \times \mathbb R$.
Take positive elements $((1,0),u_1),((0,1),u_2),((0,1),v_1),((1,0),v_2)$ from $(\mathbb R \times \mathbb R)\lex G$ such that $((1,0),u_1)+((0,1),u_2)=((0,1),v_1)+((1,0),v_2)$. We can find an RDP table
as follows
$$
\begin{matrix}

((0,1),v_1) &\vline & (e_{11}, c_{11}) & (e_{12},c_{12})\\
((1,0),v_2) &\vline & (e_{21},c_{21}) & (e_{22},c_{22})\\
  \hline     &\vline      &((1,0),u_1) & ((0,1),u_2)
\end{matrix}\ \ .
$$
Since $(0,1)\wedge (1,0)=(0,0)$ in $\mathbb R\times \mathbb R$, we have $e_{11}=e_{22}=(0,0)$. Therefore, $c_{11}\ge 0$ and $c_{22}\ge 0$.
\end{proof}

\begin{rmk}\label{re:RDP17}
(1) We note that there is an Abelian po-group with RDP that does not satisfy the conditions (P1)--(P2) of Proposition \ref{pr:equal}. Let $G=\{(x,y)\in \mathbb R\times \mathbb R\mid x+y=0\}$ be a po-subgroup of $\mathbb R\times \mathbb R$ ordered with respect to the original order in $\mathbb R\times \mathbb R$. Then $G^+=\{(0,0)\}$, so $G$ is an Abelian po-group satisfying RDP. But $G$ is not directed, so by (iii) of Theorem \ref{th:RDP13'}(iii) and Proposition \ref{pr:equal}, it does not satisfies (P1)--(P2).

(2) Similarly, every non-directed po-group does not satisfy (P1)--(P2).
\end{rmk}

(3) Every directed Abelian po-group, $G$, satisfies (P1)--(P2). Indeed, let $u_1+u_2=v_1+v_2$ hold. Due to directness of $G$, there is an element $k \in G$ such that $k\ge 0, v_2-u_1$. Then $v_2\le u_1+k$ and of course, $u_2-k$ and $v_2-k$ commute.

\vspace{2mm}
\noindent
{\bf Question.} Does there exist a directed po-group with RDP and not with wRDP, or does every directed po-group with RDP satisfies wRDP?

\vspace{2mm}
A partial answer to this question is the following example.

\begin{exm}\label{ex:5.8}
There is a directed po-group with RDP$_0$ and not satisfying RDP which does not satisfy wRDP.
\end{exm}

\begin{proof}
We use an example from Remark \ref{rm:RDP}. Thus let $G$ be an additive group generated freely by the countably many elements $g_0, g_1, \ldots$, let $v: (G;+,0)\to (\mathbb R;+,0)$ be the homomorphism determined by the conditions $v(g_{2i} )=v(g_{2i+1}) = (1/2)^i$, $i =0,1,\ldots$. Define a partial order in $G$ by setting
$G^+:= \{x \in  G\mid x = 0 \mbox{ or } v(x) > 0\}$. This means that we have for $a, b \in G$ $a\le b$ iff $a=b$ or $v(a)<v(b)$. In \cite[Ex 3.6]{DvVe1}, there was proved that $G$ satisfies RDP$_0$ but RDP fails.

We assert that for this $G$, conditions (P1)--(P2) fail.
Take elements $g_1,g_2,g_3$ and we put $u_1 = g_3-g_1$, $u_2= g_1$, $v_1=g_3-g_2$, and $v_2=g_2$. Then $u_1+u_2=v_1+v_2$. We show that there is no $k\ge 0$ such that $v_2\leq u_1+k$, and $u_2-k$ and $v_2-k$ commute. From  construction of $G$, we  see that $k$ has to be strictly positive. Then we have to verify
$$
g_2\le (g_3-g_1)+k \quad {\rm and} \quad g_1-k+g_2-k= g_2-k+g_1-k. \eqno(5.1)
$$

Assume that conditions (P1)--(P2) hold for this case. Then  $u_2-k$ and $v_2-k$ commute, i.e. $g_2-k+g_1=g_1-k+g_2$.

To prove a contradiction, we use the word techniques.
The free group with generators $g_0,g_1, \ldots$ can be identified with the set of reduced words $n_1g_{n_1}+\cdots + n_lg_{n_l}$ for $n_1,\ldots,n_l \in \{-1,1\}$ over the alphabet $g_0, -g_0, g_1, -g_1,\ldots$, where reduced means that there are no successive letters $g_i, -g_i$ or $-g_i, g_i$ in the word.

Take two words $g_2-k+g_1$ and $g_1-k+g_2$ which  are identifiable with the same element and
thus, in particular, of the same length. Hence either both are reduced or
both are not reduced. Let $k=k_1+\cdots + k_m$ be the reduced word. Comparing the same words $g_2 - k_m-\cdots-k_1 +g_1$ and $g_1 - k_m-\cdots-k_1 +g_2$, we see that both words are not reduced, hence,  $k=n(g_1-g_2)+g_1$ for $n\ge 0$, or $k = n(g_2-g_1)+g_2$ for $n \ge 0$. Since $v(g_3)= 1/4$ and $v(g_1)=v(g_2)=1/2$, for $k=n(g_1-g_2)+g_1$, we have $(g_3 -g_1)+k = (g_3-g_1)+g_1+n(-g_2-g_1)= g_3 + n(-g_2+g_1) <g_2$ while  $1/3 = v(g_3+n(-g_2+g_1)) < 1/2=v(g_2)$ which contradicts (5.1). Similarly, if $k =n(g_2-g_1)+g_2$, we have $v(g_3-g_1+n(g_2-g_1)+g_2)=v(g_3)<v(g_2)$, i.e. $g_3<g_2$ which also contradicts (5.1).

Hence, (P1)--(P2) fail in $G$.
\end{proof}

The latter example can be generalized as follows.

\begin{exm}\label{ex:5.9}
Let $G$ be an additive group generated freely by the countably many elements $g_0, g_1, \ldots$, let $v: (G;+,0)\to (\mathbb R;+,0)$ be the homomorphism determined by the conditions $v(g_i )>0$, $i =0,1,\ldots$ and $\liminf_n v(g_n)=0$. Define a partial order in $G$ by setting
$G^+:= \{x \in  G\mid x = 0 \mbox{ or } v(x) > 0\}$. Then $G$ is a directed po-group with RDP$_0$ but RDP and sRDP fail in $G$.

%If in addition, there are two different generators $g_i$ and $g_j$ in $G$ such that $v(g_i)=v(g_j)$, then also wRDP fails in $G$.
\end{exm}

\begin{proof}
The range of $v$, $v(G)$ is a subgroup of $\mathbb R$.
Since $\liminf_n v(g_n)=0$, we have that $v(G)$ is dense in $\mathbb R$, see \cite[Lem 4.21]{Goo}. By \cite[Ex. 10]{Fuc1}, $G$ has RIP, which by Lemma \ref{le:RIP} gets $G$ has RDP$_0$.

To exhibit RDP, take $a_1,a_2,b_1,b_2 \in G^+$ such that $a_1+a_2=b_1+b_2$. Without loss of generality we can assume $a_1,a_2,b_1,b_2 >0$.  For it we search an RDP table in the form

$$
\begin{matrix}
a_1  &\vline & a_1-k & k\\
a_{2} &\vline & k-a_1+b_1  & -k+b_2\\
  \hline     &\vline      &b_{1} & b_{2}
\end{matrix}\ \ ,
$$
where $k\ge 0$.
Hence, $k-a_1+b_1=a_2-b_2+k$. If $v(a_1) <v(b_1)$, we put $k=0$. If $v(a_1)>v(b_1)$, we put $k = -b_1+a_1$. If $v(a_1)=v(b_1)$, the case $a_1=b_1$ is trivial, so let $a_1\ne b_1$. We claim that for such a case, there is no $k\ge 0$. Indeed, let $a$ be any element of $G$ such that $v(a)=0$. Choose $b_1=a_1+a$, $b_2=-a+a_2$. Then $k$ has to commute with $a$, in particular, $k$ has to commute with $a=g_1+g_2-g_1-g_2$. Using word technique, we show that then $k=n(g_1+g_2-g_1-g_2)$ for some $n\in \mathbb Z$. Indeed, let $k=k_1+\cdots + k_m$  be a reduced word and take two same words $k_1+\cdots + k_m+ g_1+g_2-g_1-g_2$ and $g_1+g_2-g_1-g_2+k_1+\cdots + k_m$. They are simultaneously reduced or non-reduced.

Let the words be reduced. It is possible to show that $m= 4j$. Comparing letters, we have $k_1=g_1=k_{m-3}$, $k_2=g_2=k_{m-2}$, $k_3=-g_1=k_{m-1}$, $k_4=-g_2=k_m$. Hence
$k_5+\cdots + k_{m-4}+ g_1+g_2-g_1-g_2= g_1+g_2-g_1-g_2 + k_5+\cdots + k_{m-4}$, which gives after finitely many cases $k = n(g_1+g_2-g_1-g_2)$ for some $n\ge 1$.

Let the words be not reduced, then $k_m=-g_1$ and $k_1= g_2$, and similarly we can show that then $k=n(g_2+g_1-g_2-g_1)$ for some $n \ge 0$.

Since $v(k)=nv(g_1+g_2-g_1-g_2)=0$, then $k \ge 0$ only if $k=0$ which implies that $v(k-a_1+b_1)=0$ but $a_1+b_1$ is not positive. Hence, RDP fails in $G$.

To prove that (P1)--(P2) fail, let $u_1+u_2=v_1+v_2$ be given, where $u_2=g_1+g_2$, $v_2=g_2+g_1$, and $u_1<0$. Then also $v_1<0$. Using word technique, to find a solution for $u_2-k+v_2=v_2-k+u_2$,  we exhibit words $-g_1- g_2+k_1+\cdots +k_m- g_2-g_1= -g_2- g_1+k_1+\cdots +k_m- g_1-g_2$. The equation has a solution only if $m=4i+2$, and then $k=n(g_1+g_2-g_1-g_2)+g_1+g_2$ or $k=n(g_2+g_1-g_2-g_1)+g_2+g_1$ for $n\ge 0$. Check $v(u_1+k)=v(u_1)+v(n(g_1+g_2-g_1-g_2)) + v(g_1+g_2)< v(g_1+g_2)=v(v_2)$ which entails $u_1+k<v_2$. The same is true for the second solution of $k$. Hence, wRDP fails in $G$.
\end{proof}

If in Example \ref{ex:5.9}, we assume that $G$ is an Abelian group freely generated by $g_1,g_2,\ldots$ and the order is the same as in Example \ref{ex:5.9}, then $G$ is a directed po-group with RDP and with wRDP.

%\section{Lexicographic Pseudo Effect Algebras}
\section{Conclusion}%6

Let $\mathcal{LRDP}$ be the class of po-groups $A$ with RDP such that $A\lex G$ has RDP for each directed po-group $G$ with RDP. We have shown that the class $\mathcal{LRDP}$ contains all
\begin{itemize}
\item[{\rm (i)}] linearly ordered groups, \cite[Thm 3.1]{262}, Theorem \ref{th:RDP3};
\item[{\rm (ii)}] antilattice po-groups with RDP, Theorem \ref{th:3.3};
\item[{\rm (iii)}] direct products of linearly ordered Abelian groups satisfying the condition of Theorem \ref{th:RDP3} and with the strict product ordering, Theorem \ref{th:RDP3};
%\item[{\rm (iv)}] direct product of antilattice po-groups with RDP and with the strict product ordering, Theorem \ref{th:RDP4};

\item[{\rm (iv)}] po-groups with RDP satisfying NCDP Theorem \ref{th:RDP13'}(i).

%\item[{\rm (v)}] Abelian po-groups with RDP WHY?
\end{itemize}

In the paper we have obtained other interesting conditions when the lexicographic product of two po-groups has RDP.

It is still an open problem whether e.g. every po-group with RDP belongs $\mathcal{LRDP}$ or a weaker problem whether every Abelian po-group with RDP belongs to $\mathcal{LRDP}$.

The study of the lexicographic product of po-groups is important for the study of so-called lexicographic pseudo effect algebras, i.e. when we can represent a pseudo effect algebra as an interval in a unital po-group $(H\lex G,(u,0))$, where $(H,u)$ is a unital po-group with RDP and $G$ is a directed po-group with RDP. The authors hope to continue in these applications of the lexicographic product of po-groups for pseudo effect algebras.

%%%%%%%%%%%%%%%%%%%%%%%%%%%%%%%%%%%%%%%%%%%%%%%%
%%%%%%%%%%%%%%%%%%%%%%%%%%%%%%%%%%%%%%%%%%%%%%%%

\end{document}